\documentclass[11pt]{amsart}
\usepackage{amsmath,amsfonts,amscd,amssymb}
\theoremstyle{plain}
\newcounter{thmcount}
\newtheorem{theorem}[thmcount]{Theorem}
\newtheorem{conjecture}[thmcount]{Conjecture}
\newtheorem{lemma}[thmcount]{Lemma}
\theoremstyle{definition}
\newtheorem{remark}[thmcount]{Remark}
\newtheorem*{definition}{Definition}
\newtheorem*{notation}{Notation}
\newtheorem*{example}{Example}
\newtheorem*{acknowledgements}{Acknowledgements}

\input cyracc.def
\font\tencyr=wncyr10
\def\sha{\text{\tencyr\cyracc{Sh}}}

\def\F{{\mathbb F}}

\def\Q{{\mathbb Q}}

\def\Z{{\mathbb Z}}

\def\newmathop#1{\expandafter\gdef\csname #1\endcsname{\mathop{\rm #1}\nolimits}}
\newmathop{rk}
\newmathop{ord}
\newmathop{Gal}
\newmathop{Res}
\newmathop{Sel}
\let\oldchar\char
\newmathop{char}

\let\char\oldchar
\newmathop{GL}
\newmathop{Re}
\newmathop{Im}

\def\dimF2{\dim_{\scriptscriptstyle\F_2}\!}
\def\dimFp{\dim_{\scriptscriptstyle\F_p}}

\let\iso\cong

\def\squeezedcdots{\cdot\!\cdot\!\cdot}

\def\comment{}
\def\endcomment{}

\def\nth#1{$#1^{\rm th}$}

\begin{document}

\hskip -10cm.\vskip -5mm

\title{A note on the Mordell-Weil rank modulo \larger[3]$n$}
\author{Tim and Vladimir Dokchitser}
\address{Robinson College, Cambridge CB3 9AN, United Kingdom}
\email{t.dokchitser@dpmms.cam.ac.uk}
\address{Emmanuel College, Cambridge CB2 3AP, United Kingdom}
\email{v.dokchitser@dpmms.cam.ac.uk}

\begin{abstract}
Conjecturally,
the parity of the Mordell-Weil rank of an \hbox{elliptic} curve over a
number field $K$ is determined by its root number. The root number
is a product of local root numbers, so the rank modulo 2
is (conjecturally)
the sum over all places of $K$ of a function of elliptic curves
over local fields. This note shows that there can be no analogue
for the rank modulo 3, 4 or 5, or for the rank itself.
In fact, standard conjectures for elliptic curves imply
that there is no analogue modulo $n$ for any $n>2$,
so this is purely a parity phenomenon.
\end{abstract}

\maketitle

{\vskip-3mm\smaller[2]\hfill Surely, otherwise somebody would have spotted it by now.}
{\par\noindent\vskip-3pt\noindent\smaller[2]\hfill --- {\it Tom Fisher}\phantom{x}\\[2pt]}

\comment

It is a consequence of the Birch--Swinnerton-Dyer conjecture
that the parity of the Mordell-Weil rank of an elliptic curve $E$ over a
number field~$K$ is determined by its root number, the sign in the
functional equation of the $L$-function.
The root number is a product of local root numbers,
which leads to a conjectural formula of the form
$$
  \rk E/K \equiv \sum\limits_v\lambda(E/K_v)\mod 2,
$$
where $\lambda$ is an invariant of elliptic curves over {\em local\/} fields,
and $v$ runs over the places of $K$.
One might ask whether there is a local expression like this
for the rank modulo 3 or modulo 4, or even for the rank itself.
The purpose of this note is to show that, unsurprisingly,
the answer is `no'.

The idea is simple:
if the rank modulo $n$ were a sum of local $\Z/n\Z$-valued invariants,
then $\rk E/K$ would be a multiple of $n$ whenever
$E$ is defined over $\Q$ and $K/\Q$ is a Galois extension where every
place of $\Q$ splits into a multiple of $n$ places.
However, for small $n>2$ it is easy to find $E$ and $K$ for which this
property fails (Theorem \ref{thmmw345}).
In fact, if one believes the standard heuristics concerning
ranks of elliptic curves in abelian extensions, it fails for every $n>2$
and every $E/\Q$ (Theorems \ref{locp}, \ref{loc4}).

This kind of argument can be used to test whether a global invariant
has a chance of being a sum of local terms.
We will apply it to other standard invariants of elliptic curves
and show that
the parity of the 2-Selmer rank, the parity of the rank of the $p$-torsion
and $\dimF2\sha[2]$ modulo 4 cannot be expressed as a sum
of local terms (Theorem~\ref{other}).
Finally, we will also comment on $L$-functions all of whose local factors
are \nth{n} powers and discuss the parity of the analytic rank
for non-self-dual twists of elliptic curves (Remarks \ref{lfun},\ref{cubic}).

Our results only prohibit an expression for the rank as a {\em sum\/} of
local terms. Local data does determine the rank, see Remark \ref{determine}.

\section{Mordell-Weil rank is not a sum of local invariants}

\begin{definition}
Suppose $(K,E)\mapsto\Lambda(E/K)$ is some global invariant
of elliptic curves over number fields%
\footnote{
Meaning
that if $K\iso K'$ and
$E/K$ and $E'/K'$ are isomorphic elliptic curves (identifying $K$ with $K'$),
then $\Lambda(E/K)=\Lambda(E'/K')$.}.
We say it is a {\em sum of local invariants} if
$$
  \Lambda(E/K) = \sum\limits_v \lambda(E/K_v),
$$
where
$\lambda$ is some invariant of elliptic curves over local fields,
and the sum is taken over all places of $K$.
\end{definition}

\noindent
Implicitly, $\Lambda$ and $\lambda$ take values in some abelian group $A$,
usually $\Z$. Moreover $\lambda(E/K_v)$ should be $0$
for all but finitely many $v$.

\begin{example}
If the Birch--Swinnerton-Dyer conjecture holds
(or if $\sha$ is finite, see \cite{Kurast}), then the Mordell-Weil rank
modulo 2 is a sum of local invariants with values in $\Z/2\Z$. Specifically,
for an elliptic curve $E$ over a local field $k$ write $w(E/k)\!=\!\pm1$ for
its local root number, and define $\lambda$ by
$(-1)^{\lambda(E/k)}\!=\!w(E/k)$.
Then
$$
  \rk E/K \equiv \sum\limits_v\lambda(E/K_v)\mod 2.
$$
An explicit description of local root numbers can be found in
\cite{RohG} and \cite{Kurast}.
\end{example}

\begin{theorem}
\label{thmmw}
The Mordell-Weil rank is not a sum of local invariants.
\end{theorem}

\noindent
This is a consequence of the following stronger statement:

\begin{theorem}
\label{thmmw345}
For $n\in\{3,4,5\}$ the Mordell-Weil rank modulo $n$ is not a sum of local
invariants (with values in $\Z/n\Z$).
\end{theorem}

\begin{lemma}
\label{split}
Suppose $\Lambda: \text{(number fields)}\to\Z/n\Z$ satisfies
$\Lambda(K)=\sum_v \lambda(K_v)$ for some function
$\lambda: \text{(local fields)}\to\Z/n\Z$. Then $\Lambda(F)=0$ whenever
$F/K$ is a Galois extension of number fields in which
the number of places above each place of $K$ is a multiple of $n$.
\end{lemma}

\begin{proof}
In the local expression for $\Lambda(F)$
each local field occurs a multiple of $n$ times.
\end{proof}

%

\begin{proof}[Proof of Theorem \ref{thmmw345}.]
Take $E/\Q: y^2=x(x+2)(x-3)$, which is 480a1 in Cremona's notation.
Writing $\zeta_p$ for a primitive \nth{p} root of unity, let
$$
  F_n = \left\{\begin{array}{ll}
    \text{the degree 9 subfield of }\Q(\zeta_{13},\zeta_{103})\quad & \text{if } n=3, \cr
    \text{the degree 25 subfield of }\Q(\zeta_{11},\zeta_{241}) & \text{if }n=5, \cr
    \Q(\sqrt{-1},\sqrt{41},\sqrt{73}) & \text{if }n=4.\cr
  \end{array}
  \right.
$$
Because 13 and 103 are cubes modulo one another, and all other primes are
unramified in $F_3$, every place of $\Q$ splits into 3 or 9 in $F_3$.
Similarly $F_4$ and $F_5$ also satisfy the assumptions of Lemma \ref{split}
with $n=4, 5$. Hence, if
\pagebreak
the Mordell-Weil rank modulo $n$ were a sum of local
invariants, it would be $0\in\Z/n\Z$ for $E/F_n$.
However, 2-descent shows that $\rk E/F_3=\rk E/F_5=1$ and $\rk E/F_4=6$
(e.g. using Magma \cite{Magma}, over all minimal non-trivial subfields of $F_n$).
\end{proof}

\begin{remark}
\label{lfun}
%
The $L$-series of the curve $E\!=\,$480a1 used in the proof
over $F=F_4=\Q(\sqrt{-1},\sqrt{41},\sqrt{73})$
is formally a 4th power, in the sense that each Euler factor is:
$$
L(E/F,s) = \textstyle
       1
       \!\cdot\! \bigl( \frac {1}{1-3^{-2s}} \bigr)^4
       \bigl( \frac {1}{1-5^{-2s}} \bigr)^4
       \bigl( \frac {1}{1 + 14\cdot 7^{-2s} + 7^{2-4s}} \bigr)^4
       \bigl( \frac {1}{1 + 6\cdot 11^{-2s} + 11^{2-4s}} \bigr)^4\!\squeezedcdots.
$$
However, it is not a 4th power of an entire function, as it vanishes
to order~6 at $s=1$.
Actually,
it is not even a square of an entire function:
it has a simple zero at $1+2.1565479...\,i$.

In fact,
by construction of $F$, for any $E/\Q$ the $L$-series $L(E/F,s)$
is formally a 4th power and vanishes to even order at $s=1$ by the
functional equation. Its square root has analytic
continuation to a domain including $\Re s>\frac 32$,
$\Re s<\frac 12$ and the real axis, and satisfies a functional equation
$s\leftrightarrow 2-s$, but it is not clear whether it has an arithmetic
meaning.
\end{remark}

\begin{lemma}
\label{quad}
Suppose an invariant $\Lambda\in\Z/2^{k}\Z$ is a sum of local invariants.
Let $F=K(\sqrt{\alpha_1},...,\sqrt{\alpha_{m}})$ be a multi-quadratic extension
in which every
prime of $K$ splits into a multiple of $2^k$ primes of $F$.
Then for every elliptic curve $E/K$,
$$\Lambda(E/K)+\sum\limits_D\Lambda(E_D/K)=0,$$
where the sum is taken over the quadratic subfields $K(\sqrt D)$ of $F/K$,
and $E_D$ denotes the quadratic twist of $E$ by $D$.
\end{lemma}

\begin{proof}
In the local expression for the left-hand side of the formula
each local term ($\lambda$ of a given elliptic curve
over a given local field)
occurs a multiple of $2^k$ times.
\end{proof}

\begin{theorem}
\label{other}
Each of the following is not a sum of local invariants:
\begin{itemize}
\item $\dimF2\sha(E/K)[2]\mod 4$,                   
\item $\rk(E/K)+\dimF2\sha(E/K)[2]\mod 4$,
\item $\dimF2\Sel_2 E/K\mod 2$,
\item $\dimFp E(K)[p]\mod 2$ for any prime $p$.
\end{itemize}
Here $\sha$ is the Tate-Shafarevich group and $\Sel_2$ is the 2-Selmer group.
\end{theorem}

\begin{proof}
The argument is similar to that of Theorem \ref{thmmw345}:

For the first two claims, apply Lemma \ref{quad}
to $E: y^2+y=x^3-x$ (37a1) with $K=\Q$ and $F=\Q(\sqrt{-1},\sqrt{17},\sqrt{89})$.
The quadratic twists of $E$ by $1,-17,-89,17\cdot 89$ have rank 1, and
those by $-1,17,89,-17\cdot 89$ have rank~0; the twist by $-17\cdot 89$
has $|\sha[2]|=4$ and the other seven have trivial $\sha[2]$.
The sum over all twists is therefore 2 mod 4 in both cases, so
they are not sums of local invariants.

\pagebreak

For the parity of the 2-Selmer rank and of $\dim E[2]$ apply Lemma \ref{split}
to $E: y^2+xy+y=x^3+4x-6$ (14a1) with $K=\Q$, $F=\Q(\sqrt{-1},\sqrt{17})$
and $n=2$. The 2-torsion subgroup of $E/F$ is of order 2
and its 2-Selmer group over $F$ is of order 8.

Finally, for $\dimFp E[p]\mod 2$ for $p>2$ take any elliptic curve $E/\Q$
with $\Gal(\Q(E[p])/\Q)\iso\GL_2(\F_p)$, e.g. $E: y^2=x^3-x^2+x$ (24a4),
see \cite{SerP}~5.7.2.
Let $K$ be the field obtained by adjoining to $\Q$ the coordinates of one
$p$-torsion point and $F=K(\sqrt{-1},\sqrt{17})$. Because
$F$ does not contain the \nth{p} roots of unity, $\dimFp E(F)[p]=1$.
So, by Lemma \ref{split} the parity of this dimension
is not a sum of local invariants.
\end{proof}

\begin{remark}
\label{cubic}
The functional equation expresses the parity of the analytic rank
as a sum of local invariants not only for elliptic curves
(or abelian varieties), but also for their twists by self-dual
Artin representations. However, for the parity of the rank of
non-self-dual twists there is presumably no such expression.

For example,
let $\chi$ be a non-trivial Dirichlet character of $(\Z/7\Z)^\times$
of order~3.
Then there is no function $(k,E)\mapsto\lambda(E/k)\in\Z$
defined for elliptic curves over local fields $k$,
such that for all elliptic curves $E/\Q$,
$$
  \ord_{s=1} L(E,\chi,s) \equiv \smash{\sum_v} \lambda(E/\Q_v) \mod 2.
$$
To see this, take
$$
  E/\Q: y^2+y=x^3+x^2+x \text{ (19a3)}, \quad
  K=\Q, \quad F=\Q(\sqrt{-1},\sqrt{17})
$$
and apply Lemma \ref{quad}.
The twists of $E, E_{-1}$ and $E_{17}$ by $\chi$ have analytic rank~0,
and that of $E_{-17}$ has analytic rank 1, adding up to an odd number.
\end{remark}


\section{Expectations}

We expect the Mordell-Weil rank modulo $n$ not to be a sum of local terms
for any $n>2$ and any class of elliptic curves.
Theorems \ref{locp} and \ref{loc4} below show that this
is a consequence of modularity of elliptic curves, the known cases of
the Birch--Swinnerton-Dyer conjecture and standard conjectures for
analytic ranks of elliptic curves.

\begin{notation}
For a prime $p$ we write $\Sigma_p$ for the set of all Dirichlet characters
of order $p$. We say that $S\subset\Sigma_p$ has density $\alpha$ if
$$
  \lim_{x\to\infty}
    \frac{ \{\chi: \chi\in S \>|\> N(\chi)<x\} }
         { \{\chi: \chi\in \Sigma_p \>|\> N(\chi)<x\} } = \alpha,
$$
where $N(\chi)$ denotes the conductor of $\chi$.
\end{notation}

\begin{conjecture}[Weak form of \cite{DFK2} Conj. 1.2]
\label{conjmin}
For $p>2$ and every elliptic curve $E/\Q$,
those $\chi\in \Sigma_p$ for which $L(E,\chi,1)=0$
have density 0 in $\Sigma_p$.
\end{conjecture}

\pagebreak

\begin{theorem}
\label{locp}
Let $E/\Q$ be an elliptic curve and $p$ an odd prime.
Assuming Conjecture \ref{conjmin},
there is no function $k\mapsto\lambda(E/k)\in\Z/p\Z$
defined for local fields $k$ of characteristic 0, such that
for every number field~$K$,
$$
  \rk E/K \equiv \sum_v \lambda(E/K_v) \mod p.
$$
\end{theorem}

\begin{lemma}
\label{density}
Let $p$ be a prime number and $S\subset \Sigma_p$ a set of characters
of density 0 in $\Sigma_p$. For every $d\ge 1$ there is an abelian extension $F_d/\Q$
with Galois group $G\iso \F_p^d$, such that no characters of $G$ are in $S$.
\end{lemma}

\begin{proof}
Without loss of generality, we may assume that if $\chi\in S$ then
$\chi^n\in S$ for $1\le n<p$.
When $d=1$, take $F_1$ to be the kernel of any $\chi\in \Sigma_p\setminus S$.
Now proceed by induction, supposing that $F_{d-1}$ is constructed.
Writing $\Psi$ for the set of characters of $\Gal(F_{d-1}/\Q)$, the set
$$
  S_d = \bigcup_{\psi\in\Psi} \{\phi\psi: \phi \in S\}
$$
still has density 0. Pick any $\chi\in \Sigma_p\setminus S_d$, and set $F_d$
to be the compositum of $F_{d-1}$ and the degree $p$ extension of $\Q$
cut out by $\chi$.
It is easy to check that no character of $\Gal(F_d/\Q)$ lies in $S$.
\end{proof}

\begin{proof}[Proof of Theorem \ref{locp}.]
Pick a quadratic field $\Q(\sqrt D)$ such that
the quadratic twist $E_D$ of $E$ by $D$ has analytic
rank 1, which is possible by \cite{BFH,MM,Wal}.
By Conjecture \ref{conjmin},
the set $S$ of Dirichlet characters $\chi$ of order $p$ such that
$L(E_D,\chi,1)=0$ has density 0. Apply Lemma \ref{density}
to $S$ with $d=3$. Every place of $\Q$ splits in the resulting field $F=F_3$
into a multiple of $p$ places ($\Q_l$~has no $\F_p^3$-extensions, so every
prime has to split).

Arguing by contradiction, suppose the rank of $E$ mod $p$ is a sum
of local invariants. Because in $F/\Q$ and therefore also in
$F(\sqrt D)/\Q$ every place
splits into a multiple of $p$ places,
$$
  \rk E/F \equiv 0\mod p \quad\text{and}\quad   \rk E/F(\sqrt D) \equiv 0\mod p
$$
by Lemma \ref{split}.
Therefore $\rk E_D/F=\rk E/F(\sqrt D)-\rk E/F$ is a multiple of~$p$.
On the other hand,
$$
  L(E_D/F,s) = \prod_{\chi} L(E_D,\chi,s),
$$
the product taken over the characters of $\Gal(F/\Q)$. By construction,
it has a simple zero at $s=1$.
Because $F$ is totally real of odd degree over~$\Q$,
by Zhang's theorem \cite{ZhaH} \hbox{Thm.\/ A},
$E_D/F$ has Mordell-Weil rank $1\not\equiv 0\mod p$,
a contradiction.
\end{proof}

\begin{conjecture}[Goldfeld \cite{Gol}]
\label{goldfeld}
For every elliptic curve $E/\Q$,
those $\chi\in \Sigma_2$ for which $\ord_{s=1}L(E,\chi,s)>1$ have
density 0 in $\Sigma_2$.
\end{conjecture}

\begin{conjecture}
\label{rank2}
Every elliptic curve $E/\Q$ has a quadratic twist of Mordell-Weil rank 2.
\end{conjecture}

\begin{theorem}
\label{loc4}
Let $E/\Q$ be an elliptic curve.
Assuming Conjectures \ref{goldfeld} and~\ref{rank2},
there is no function $k\mapsto\lambda(E/k)\in\Z/4\Z$
defined for local fields $k$ of characteristic 0, such that
for every number field~$K$,
$$
  \rk E/K \equiv \sum_v \lambda(E/K_v) \mod 4.
$$
\end{theorem}

\begin{proof}
Let $\Q(\sqrt D)$ be a quadratic field such that
the quadratic twist $E'\!=\!E_D$ of $E$ by $D$ has Mordell-Weil rank 2
(Conjecture \ref{rank2}).
The set $S$ of those $\chi\in \Sigma_2$ for which
$\ord_{s=1}L(E',\chi,s)>1$ has density~0 (Conjecture \ref{goldfeld}).
Let~$P$ be the set of primes where $E'$ has bad reduction union $\{\infty\}$,
and apply Lemma \ref{density} to $S$ with $d=5+3|P|$.
The resulting field $F_d$ has a subfield $F$ of degree $2^5$ over $\Q$,
where all places in $P$ split completely:
$\Q_l$ has no $\F_2^4$-extensions, so the condition that a given place in $P$
splits completely drops the dimension by at most 3.
By the same argument, every place of $\Q$ splits in $F$
into a multiple of $4$ places.

Arguing by contradiction, suppose the rank of $E$ mod $4$ is a sum
of local invariants.
Because in $F/\Q$ and therefore also in
$F(\sqrt D)/\Q$ every place
splits into a multiple of $4$ places,
$$
  \rk E/F \equiv 0\mod 4 \quad\text{and}\quad   \rk E/F(\sqrt D) \equiv 0\mod 4
$$
by Lemma \ref{split}.
Therefore $\rk E'/F=\rk E/F(\sqrt D)-\rk E/F$ is a multiple of~$4$.
Now we claim that $E'/F$ has rank 2 or 33, yielding a contradiction.

Let $\Q(\sqrt m)\subset F$ be a quadratic subfield. The root number
of $E'$ over $\Q(\sqrt m)$ is 1, because the root number is a product of
local root numbers and the places in $P$ split in $\Q(\sqrt m)$.
(The local root number is $+1$ at primes of good reduction.) So
$$
  L(E'/\Q(\sqrt m),s) =  L(E'/\Q,s) L(E'_m/\Q,s)
$$
vanishes to even order at $s=1$. Hence the 31 twists
of $E'$ by the non-trivial characters of $\Gal(F/\Q)$
have the same analytic rank 0 or 1, by the choice of~$F$.
By Kolyvagin's theorem \cite{Kol}, their Mordell-Weil ranks
are the same as their analytic ranks,
and so $\rk E'/F$ is either $2+0$ or $2+31$.
\end{proof}

\endcomment

\begin{remark}
In some cases, it may seem reasonable to try and write some
global invariant in $\Z/n\Z$ as a sum of local invariants in $\frac 1m\Z/n\Z$,
i.e. to allow denominators in the local terms. For instance,
one could ask whether the parity of the rank of a cubic twist
(as in Remark \ref{cubic}) can
be written as a sum of local invariants of the form $\frac a3\!\!\mod 2\Z$.

However, introducing a denominator does not appear to help.
First, the prime-to-$n$ part $m'$ of $m$ adds no flexibility, as can be seen
by multiplying the formula by $m'$. (For instance, if there were a
formula for the parity of the rank of a cubic twist as a sum of
local terms in $\frac a3\!\!\mod 2\Z$, then multiplying it by 3 would
yield a formula for the same parity with local terms in $\Z/2\Z$.)
As for the non-prime-to-$n$ part, e.g. the proofs of Theorems~
\ref{locp} and \ref{loc4} immediately adapt to local invariants in
$\frac{1}{p^k}\Z/p\Z$ and $\frac{1}{2^k}\Z/4\Z$, by increasing $d$
by $k$.
\end{remark}

\begin{remark}
\label{determine}
The negative results in this paper rely essentially on the fact that we
allow only additive formulae for global invariants in terms of local
invariants. Although Theorem 1 shows that there is no formula
of the form
$$
  \rk E/K = \sum_v \lambda(E/K_v),
$$
the Mordell-Weil rank {\em is\/} determined by the set $\{E/K_v\}_v$
of curves over local fields.
In other words, 
$$
  \rk E/K = \text{function}(\{E/K_v\}_v).
$$
In fact, for any abelian variety $A/K$ the set $\{A/K_v\}_v$ determines
the $L$-function $L(A/K,s)$ which is the same as $L(W/\Q,s)$ where
$W$ 
is the Weil restriction of $A$ to $\Q$. By Faltings'
theorem \cite{Fa} the $L$-function recovers $W$ up to isogeny, and hence
also recovers the rank $\rk A/K\,(=\!\rk W/\Q)$.
%
\end{remark}

\begin{acknowledgements}
The first author is supported by a Royal Society
University Research Fellowship.
The second author would like to thank Gonville \& Caius College, Cambridge.
\end{acknowledgements}



\end{document}